\documentclass[12pt,a4paper,english]{article}
\usepackage[T1]{fontenc}
\usepackage[latin9]{inputenc}
\usepackage{amssymb}

\makeatletter

\usepackage{amsthm}
\usepackage{fullpage}
\usepackage{verbatim}
\usepackage{makeidx}
\usepackage{amssymb}
\usepackage{amsfonts}
\usepackage{latexsym}
\makeindex

\theoremstyle{definition}
\newtheorem{lemma}{Lemma}[section]
\newtheorem{theorem}[lemma]{Theorem}
\newtheorem{corollary}[lemma]{Corollary}
\newtheorem{proposition}[lemma]{Proposition}
\newtheorem{definition}[lemma]{Definition}

\newtheorem{remark}[lemma]{Remark}

\makeatother

\usepackage{babel}

\begin{document}
\centerline{\Large\bf $\mathbb{Z}$-graded weak modules and regularity}

\vspace{0.5cm}

\centerline{ Chongying Dong \footnote{Supported by NSF grants, and a Faculty research grant from the University of California at Santa Cruz. }\   and \ \  Nina Yu}
\centerline{Department of Mathematics, University of California, Santa Cruz, CA 95064 }

\hspace{1cm}

\centerline{\bf { Abstract }}

It is proved that if any $\mathbb{Z}$-graded weak module for vertex
operator algebra $V$ is completely reducible, then $V$ is rational
and $C_{2}$-cofinite. That is, $V$ is regular. This gives a natural
characterization of regular vertex operator algebras.

\section{Introduction}

Rationality, $C_{2}$-cofiniteness and regularity are probably the
three most important concepts in representation theory of vertex operator
algebras. Rationality, which is an analog of semisimplicity of a finite
dimensional associative algebra or Lie algebra, asserts that the admissible
module category, or $\mathbb{Z}_{+}$-graded weak module category
is semisimple {[}DLM1{]}, {[}Z{]}. Originated from the modular invariance
of trace functions in vertex operator algebra {[}Z{]}, $C_{2}$-cofiniteness
tells us that certain subspace of a vertex operator algebra has finite
codimension. Regularity, which is the strongest among the three concepts,
claims that any weak module is direct sum of irreducible ordinary modules {[}DLM1{]}.

Both rationality and $C_{2}$-cofiniteness imply that there are only finitely
many irreducible admissible modules up to isomorphism and each irreducible
admissible module is ordinary (namely, each homogeneous subspace is
finite dimensional) {[}DLM2{]}, {[}KL{]}. It follows immediately from
definitions that regularity implies rationality. It is shown that
regularity also implies $C_{2}$-cofiniteness {[}L{]}. The equivalence
of regularity and rationality together with $C_{2}$-cofiniteness
is established in {[}ABD{]}. But the connection among the three concepts
has not been understood fully.

It is evident that the definition of rationality is natural, but definition
of regularity is not natural or satisfactory. It seems more natural to define regularity
by semisimplicity of the weak module category. This is successfully
achieved in the present paper. In fact, a stronger result is obtained.
That is, if any $\mathbb{Z}$-graded weak module is completely reducible,
then the vertex operator algebra is regular.

Clearly, if any $\mathbb{Z}$-graded weak module is completely reducible,
then the vertex operator algebra is rational. So it remains to prove
that if any $\mathbb{Z}$-graded weak module is completely reducible
for a vertex operator algebra $V$, then $V$ is $C_{2}$-cofinite.
It is well known that the graded dual $V'$ of $V$ is also a $V$-module
{[}FHL{]}. By a result from {[}L{]}, if $L(0)$ is semisimple on the
unique maximal weak module inside the completion of $V'$, then $V$
is $C_{2}$-cofinite. The main idea is to use the universal enveloping
algebra $\mathcal{U}(V)$ introduced in {[}FZ{]} to prove that any
weak module generated by a single vector is a quotient of a $\mathbb{Z}$-graded
weak module and that $L(0)$ acts semisimply on any irreducible submodule
of a $\mathbb{Z}$-graded weak module. It is worthy to point out that
this is perhaps the first time that the universal enveloping algebra
is to play a crucial role in the representation theory.

The paper is organized as follows. In Section 2, we review different
notions of modules and $C_{2}$-cofiniteness, rationality, and regularity
from {[}FLM{]}, {[}Z{]}, {[}DLM1{]}, {[}DLM2{]}. We also recall various
results from {[}L{]}, {[}KL{]}, {[}ABD{]} on rationality, regularity
and $C_{2}$-cofiniteness. Section 3 is devoted to the universal enveloping
algebra $\mathcal{U}(V)$, essentially following from {[}FZ{]}. We
show that any weak module generated by a single vector is a quotient
of a $\mathbb{Z}$-graded weak module. We establish the main theorem
in Section 4. That is, if any $\mathbb{Z}$-graded weak $V$-module
is completely reducible, then $V$ is $C_{2}$-cofinite, and consequently
regular.

\section{Preliminary}

Throughout this paper, we assume that $V$ is a simple vertex
operator algebra (cf. {[}B{]}, {[}FLM{]}). We first recall notions
of weak, admissible and ordinary modules from {[}FLM{]}, {[}Z{]},
{[}DLM1{]}.

\begin{definition}Let $(V,\ Y,\ 1,\ \omega)$ be a vertex operator
algebra. A \emph{weak module $M$ }for $V$ is a vector space equipped
with a linear map\[
Y_{M}:\ V\to(End\ M)[[z,\ z^{-1}]]\]
\[
v\mapsto Y_{M}(v,\ z)=\sum_{n\in\mathbb{Z}}v_{n}z^{-n-1},\ v_{n}\in End\ M\]
satisfying the following conditions:

1. $u_{m}w=0,$ for $u\in V,\ w\in M,$ $m\in\mathbb{Z}$ and $m$
large enough;

2. $Y_{M}(1,\ z)=Id_{M};$

3. For any $u,\ v\in V,$ \[
\begin{array}{c}
z_{0}^{-1}\delta\left(\frac{z_{1}-z_{2}}{z_{0}}\right)Y_{M}(u,\ z_{1})Y_{M}(v,\ z_{2})-z_{0}^{-1}\delta\left(\frac{z_{2}-z_{1}}{-z_{0}}\right)Y_{M}(v,\ z_{2})Y_{M}(u,\ z_{1})\\
=z_{2}^{-1}\delta\left(\frac{z_{1}-z_{0}}{z_{2}}\right)Y_{M}(Y(u,\ z_{0})v,\ z_{2}).\end{array}\]

\end{definition}

\begin{definition} An \emph{admissible $V$-module }$M$ is a weak
$V$-module which carries a $\mathbb{Z}_{+}$-grading $M=\oplus_{n\in\mathbb{Z}_{+}}M(n)$
satisfying the following condition: If $r,\ m\in\mathbb{Z},$ $n\in\mathbb{Z}_{+}$
and $u\in V_{r}$, then $u_{m}M(n)\subset M(r+n-m-1)$. \end{definition}

\begin{definition} An \emph{ordinary }$V$-\emph{module} $M$ is
a weak $V$-module such that $M=\oplus_{\lambda\in\mathbb{C}}M_{\lambda}$,
$\dim M_{\lambda}<\infty$ and $M_{\lambda+n}=0$ for fixed $\lambda\in\mathbb{C}$
and small enough $n\in\mathbb{Z},$ where $L(0)M_{\lambda}=\lambda$
and $Y_{M}(\omega,\ z)=\sum_{m\in\mathbb{Z}}L(m)z^{-m-2}$.

\end{definition}

It is easy to see that an ordinary module is always admissible.

\begin{definition} A vertex operator algebra $V$ is said to be \emph{rational
}if the admissible module category is semisimple.\end{definition}

We notice that the rationality is stronger than the complete reducibility of 
any admissible $V$-module as a submodule of a $\mathbb{Z}_+$-graded weak module is not necessarily 
$\mathbb{Z}_+$-graded.

It is proved in {[}DLM2{]} that if $V$ is rational, then there are
only finitely many irreducible admissible modules up to isomorphism
and any irreducible admissible module is ordinary. As pointed out
in the introduction, the $C_{2}$-cofiniteness (define below) also
implies the same results {[}KL{]}.

\begin{definition} A vertex operator algebra $V$ is said to be \emph{regular
}if any weak $V$-module $M$ is a direct sum of irreducible ordinary $V$-modules.
\end{definition}

\begin{definition} A vertex operator algebra $V$ is called $C_{2}$\emph{-cofinite
}if $\dim V/C_{2}(V)<\infty$, where $C_{2}(V)=\left\langle u_{-2}v\ |\ u,\ v\in V\right\rangle $.
\end{definition}

It is shown in {[}KL{]} and {[}ABD{]} that the regularity is equivalent
to rationality and $C_{2}$-cofiniteness. While the $C_{2}$-cofiniteness
does not imply rationality (cf. {[}A1{]}), it is widely believed that
rationality implies $C_{2}$-cofiniteness. One of the most important
conjectures in the theory of vertex operator algebras is that rationality
and regularity are equivalent. Many well known rational vertex operator
algebras such as lattice type vertex operator algebras {[}B{]}, {[}FLM{]},
{[}D{]}, {[}A2{]}, {[}DJL{]}, vertex operator algebras associated
to the integrable highest weight modules for affine Kac-Moody algebras
{[}FZ{]}, and vertex operator algebras associated to the minimal series
for Virasoro algebras {[}DMZ{]}, {[}W{]} are regular {[}DLM1{]}. Code
and framed vertex operator algebras are also regular
{[}M{]}, {[}DGH{]}.

Although the notion of regularity is very useful in the proof of modular
invariance of trace functions and Verlinde formula {[}Z{]}, {[}DLM3{]},
{[}H{]}, the definition is not natural at all as the irreducible objects
in the weak module category are assumed to be ordinary. The main purpose
of this paper is to define regularity naturally. For this we need
the following definition.

\begin{definition} An weak $V$-module $M$ is called \emph{$\mathbb{Z}$-graded
}if $M=\oplus_{i\in\mathbb{Z}}M(i)$ and for $r,\ m,\ n\in\mathbb{Z}$,
$v\in V_{n}$, $v_{r}M(m)\subset M(m+n-r-1)$. \end{definition}

We show in this paper that the regularity is equivalent to the complete reducibility of
any $\mathbb{Z}$-graded weak module. As a result, the regularity
is also equivalent to the semisimplicity of weak module category.

\section{Universal enveloping algebra}

The universal enveloping algebra $\mathcal{U}(V)$ for a vertex operator
algebra $V$ was defined in {[}FZ{]} in connection with Zhu's algebra
$A(V)$. We will extensively use $\mathcal{U}(V)$ to study the $\mathbb{Z}$-graded
weak $V$-modules in this paper. For the purpose of this paper, we
will modify the definition of universal enveloping algebra slightly
with the help of universal enveloping algebra of Lie algebras.

Let $V$ be a vertex operator algebra and $t$ an indeterminate. Consider
the tensor product

\[
\mathcal{L}(V)=\mathbb{C}[t,\ t^{-1}]\otimes V\]
Since $\mathbb{C}[t,\ t^{-1}]$ is a vertex algebra such that

\[
Y(f(t),\ z)g(t)=f(t+z)g(t)=(e^{z\frac{d}{dt}}f(t))g(t)\]
{[}B{]}, $\mathcal{L}(V)$ is a tensor product of vertex algebras
(cf. {[}DL{]}, {[}FHL{]}).

The $D$ operator of $\mathcal{L}(V)$ is given by $D=\frac{d}{dt}\otimes1+1\otimes L(-1),$
and

\[
L(V)=\mathcal{L}(V)/D\mathcal{L}(V)\]
 carries the structure of Lie algebra with bracket

\[
[u+D\mathcal{L}(V),\ v+D\mathcal{L}(V)]=u_{0}v+D\mathcal{L}(V)\]
 {[}B{]}. We use $a(n)$ to denote the image of $t^{n}\otimes a$
in $L(V)$. Then $\mathcal{L}(V)$ has a $\mathbb{Z}$-gradation given
by for homogeneous $a\in V$

\[
\deg a(n)=\mbox{wt}a-n-1\]
 As $D$ increases degree by 1 then $D\mathcal{L}(V)$ is a graded
subspace of $\mathcal{L}(V),$ so there is a naturally induced $\mathbb{Z}$-gradation
on $L(V).$

Let $U(L(V))$ be the universal enveloping algebra of $L(V).$ The
$\mathbb{Z}$-gradation on $L(V)$ induces naturally a $\mathbb{Z}$-gradation
on $U(L(V))$ so that for homogeneous $a^{i}\in V$

\[
deg(a^{1}(i_{1})\cdots a^{n}(i_{n}))=\sum_{k=1}^{n}(\mbox{wt}a^{k}-i_{k}-1).\]
Then $U(L(V))=\oplus_{k=-\infty}^{\infty}U(L(V))(k),$ where $U(L(V))(k)$
denotes the subspace consisting of the elements of degree $k,$ and
$U(L(V))(m)\cdot U(L(V))(n)\subset U(L(V))(m+n).$ We set \[
U(L(V))^{k}(n)=\sum_{i\le k}U(L(V))(n-i)\cdot U(L(V))(i).\]
Then $\{U(L(V))^{k}(n)\ |\ k\in\mathbb{Z}\}$ forms a fundamental
neighborhood system of $U(L(V))(n).$ Denote by $\widetilde{U}(L(V))(n)$
its completion. Then the direct sum $\widetilde{U}(L(V))=\sum_{n\in\mathbb{Z}}\widetilde{U}(L(V))(n)$
is a complete topological ring.

Denote by \[
J=\{\sum_{i=0}^{\infty}(-1)^{i}{r \choose i}(u(r+m-i)v(n+i)-(-1)^{r}v(r+n-i)u(m+i))\]
\[
-\sum_{i=0}^{\infty}{m \choose i}(u_{r+i}v)(m+n-i)\ |\ m,\ n,\ r\in\mathbb{Z},\ u,\ v\in V\}\]
 the subset of $\widetilde{U}(L(V))$. That is, $J$ is the set of
Jacobi relations. Define $\mathcal{U}(V),$ the \emph{universal enveloping
algebra }of $V,$ as the quotient of $\widetilde{U}(L(V))$ modulo
the two sided ideal generated by $J$. Clearly, the gradation in $\widetilde{U}(L(V))$
induces a gradation in $\mathcal{U}(V).$

Let $X=\left\langle v^{i}\ |\ i\in I\right\rangle $ be a basis
of $V$ consisting of homogeneous vectors, where $I$ is some countable
index set. For any $\overline{m}=(m_{i})_{i\in I},$ $m_{i}\in\mathbb{Z},\ $
define

\[
I_{\overline{m}}=\sum_{i\in I}\sum_{k_{i}\ge m_{i}}U(V)v^{i}(k_{i}),\]
then $I_{\overline{m}}$ is a left ideal of $\mathcal{U}(V).$

\begin{lemma} \label{U(V) weak module} $\mathcal{U}(V)/I_{\overline{m}}$
is a weak $V$-module generated by $1+I_{\overline{m}}$ such that
$a_{n}$ acts as $a(n)$ for $a\in V,\ n\in\mathbb{Z}$. \end{lemma}

\begin{proof} By the construction of $\mathcal{U}(V),$ it suffices
to prove that $v_{p}a^{1}(n_{1})\cdots a^{k}(n_{k})\in I_{M}$ for
any $v\in X,\ a^{1}(n_{1})\cdots a^{k}(n_{k})\in\mathcal{U}(V),\ p\in\mathbb{Z}$
and $p\gg0$. We prove it by induction on $k$.

If $k=1$, then for any $u,\ v\in X,$ $p,\ q\in\mathbb{Z},$ we have

\[
v_{p}u(q)=[v(p),\ u(q)]+u(q)v(p)=\sum_{i\ge0}{p \choose i}(v_{i}u)(p+q-i)+u(q)v(p).\]
Since there exists some $N\in\mathbb{N}$ such that $v_{i}u=0,\ i\ge N,$
there are only finitely many terms in the first summand of the right hand side. Hence by
definition of $I_{\overline{m}}$ we have $(v_{i}u)(p+q-i)\in I_{\overline{m}},\ p\gg0$,
$i\ge0$. Clearly $u(q)v(p)\in I_{\overline{m}}$ for $p\gg0.$ Thus
we get $v_{p}u(q)\in I_{\overline{m}}$ for $p\gg0.$

Assume that $v_{p}a^{1}(n_{1})\cdots a^{k}(n_{k})\in I_{\overline{m}}$,
$p\in\mathbb{Z},\ p\gg0$, $\forall v\in X,\ a^{1}(n_{1})\cdots a^{k}(n_{k})\in\mathcal{U}(V)$.
Then

\begin{eqnarray*} v_{p}a^{1}(n_{1})\cdots a^{k+1}(n_{k+1}) & = & [v(p), a^{1}(n_{1})]a^{2}(n_{2})\cdots a^{k+1}(n_{k+1})\\  &  &  +a^{1}(n_{1})v(p)a^{2}(n_{2})\cdots a^{k+1}(n_{k+1})\\ & = & \sum_{j\ge0}(_{j}^{p})(v_{j}a^{1})(p+n_{1}-j)a^{2}(n_{2})\cdots a^{k+1}(n_{k+1}) \\ &  &  +a^{1}(n_{1})v(p)a^{2}(n_{2})\cdots a^{k+1}(n_{k+1}). \end{eqnarray*} Similar
to the case when $k=1,$ we get $\sum_{j\ge0}(_{j}^{p})(v_{j}a^{1})(p+n_{1}-j)a^{2}(n_{2})\cdots a^{k+1}(n_{k+1})\in I_{\overline{m}},\ p\gg0.$
Also, by induction assumption, we have $v(p)a^{2}(n_{2})\cdots a^{k+1}(n_{k+1})\in I_{\overline{m}}$,
$\ p\gg0$. Therefore we get $v_{p}a^{1}(n_{1})\cdots a^{k+1}(n_{k+1})\in I_{\overline{m}},\ p\gg0$.\end{proof}

\begin{remark} We have the following results:

(1) If $\overline{m}=\bar{0},$ then $u_{n}(1+I_{\overline{m}})=0,\ \forall u\in V,\ n\ge0$.
So $1+I_{\overline{0}}$ is a vacuum-like vector. From {[}LL{]} and
the simplicity of $V$, we immediately see that $\mathcal{U}(V)/I_{\overline{m}}\cong V$.

(2) It is possible that $\mathcal{U}(V)/I_{\overline{m}}=0.$ For
example, fix $i_{0}\in I,$ let $m_{i}=0$ if $i\not=i_{0}$ and $m_{i_{0}}=-1$,
then $\mathcal{U}(V)/I_{\overline{m}}$ is a quotient of $V$ from
(1). Identify $V$ with $\mathcal{U}(V)/I_{\overline{0}}$ such that
$v=v_{-1}(1+I_{\overline{0}})$ for any $v\in V$. Then $v^{i_{0}}=0$
in $\mathcal{U}(V)/I_{\overline{m}}$ and consequently $\mathcal{U}(V)/I_{\overline{m}}=0$.
\end{remark}

\begin{remark} It is obvious that the weak module constructed in
Lemma \ref{U(V) weak module} is $\mathbb{Z}$-graded as $I_{\overline{m}}$
is a $\mathbb{Z}$-graded left ideal. \end{remark}

The importance of this kind of weak module can be seen from the following
corollary.

\begin{corollary} \label{single v. weak module Z graded} If $W$
is a weak $V$-module generated by $w\in W,$ then there exists $\overline{m}=(m_{i})_{i\in I}$
such that $W$ is a quotient of $\mathcal{U}(V)/I_{\overline{m}}$.
In particular, any weak module generated by a single vector is a quotient
of $\mathbb{Z}$-graded weak module. \end{corollary}

\begin{proof}From the definition of weak module, for any $i$, there
exists $m_{i}\in\mathbb{Z}$ such that if $n\ge m_{i},\ u_{n}^{i}w=0.$
The result follows immediately.\end{proof}

\section {Main Theorem}

The relation between $\mathbb{Z}$-graded weak modules and regularity
is investigated in this section. We will use the $\mathbb{Z}$-graded
weak module constructed in Section 3 to establish the main theorem
of this paper. We start with the following version of Schur's Lemma
for infinite dimensional representation.

\begin{lemma} \label{-Schur's}Let $A$ be an associative algebra
over $\mathbb{C}$ and $V$ a simple $A$-module of countable dimension.
Then $Hom_{A}(V,\ V)=\mathbb{C}.$ \end{lemma}

\begin{proof} A proof of this result may exist in the literature.
For the completeness of this paper, we present a short proof of this
result. Clearly, $Hom_{A}(V,\ V)$ is a division algebra over $\mathbb{C}$.
Fix $0\not=v\in V$, and define a linear map from $Hom_{A}(V,\ V)$
to $V$ such that $f$ is mapped to $f(v)$. Then the map is injective.
This implies that $Hom_{A}(V,\ V)$ has countable dimension over $\mathbb{C}$
from the assumption. Now fix $0\not=f\in Hom_{A}(V,\ V)$ and consider
the subfield $E=\mathbb{C}(f)$ of $Hom_{A}(V,\ V)$ generated by
$\mathbb{C}$ and $f$. Then $E$ has countable dimension over $\mathbb{C}.$
If $f$ is not algebraic over $\mathbb{C}$ then $E$ is isomorphic
to the field of rational functions $\mathbb{C}(x)$ where $x$ is
an indeterminate. Since $\mathbb{C}(x)$ has uncountable dimension
over $\mathbb{C}$, we have a contradiction. Consequently $f$ is
algebraic over $\mathbb{C}$ and $f\in\mathbb{C}$, the proof is completed.
\end{proof}

\begin{lemma} \label{countable dimension}Let $\overline{m}=(m_{i})_{i\in I}$.
Then the weak $V$-module $M=\mathcal{U}(V)/I_{\overline{m}}$ is
of countable dimension. In particular, any homogeneous subspace $M(n)$
has countable dimension for $n\in\mathbb{Z}$. \end{lemma}

\begin{proof} Since $M$ is generated by $1+I_{\overline{m}}$, we
have $M=\langle u_{n}(1+I_{\overline{m}})\ |\ u\in V,\ n\in\mathbb{Z}\rangle$.
Thus $M$ has countable dimension. \end{proof}

\begin{proposition} \label{Prop}Let $V$ be a vertex operator algebra
such that any $\mathbb{Z}$-graded weak $V$-module is completely reducible.
Then $L(0)$ acts semisimply on any irreducible weak $V$-module. In particular,
any irreducible weak module is $\mathbb{Z}$-graded.
\end{proposition}

\begin{proof} Let $W$ be an irreducible weak $V$-module. By Corollary
\ref{single v. weak module Z graded}, $W$ is a quotient of a weak
$V$-module $M=\mathcal{U}(V)/I_{\overline{m}}$ for some $\overline{m}.$
As pointed out in Section 3, $M=\oplus_{n\in\mathbb{Z}}M(n)$ is $\mathbb{Z}$-graded.
Since $M$ is completely reducible from the assumption, $W$ is an irreducible $\mathbb{Z}$-graded weak $V$-submodule of $M.$
For any $w\in W,$ we can write $w=w(i_{1})+\cdots+w(i_{k})$ with
$0\not=w(i_{j})\in M(i_{j})$, $1\le j\le k,$ $i_{1}<\cdots<i_{k}$.
In this case, we say $w$ is of length $k$, and denote $\ell(w)=k$.

\emph{Case 1}. If there exists $0\not=w\in W$ with $\ell(w)=1$,
i.e., there exists $w\in M(n)$ for some $n\in\mathbb{Z}.$

Set $W(m)=W\cap M(m)$, $\forall m\in\mathbb{Z}$. Since $W$ is irreducible,
the argument given in the proof of Lemma \ref{countable dimension}
shows that $W=\oplus_{m\in\mathbb{Z}}W(m)$. Note that $\mathcal{U}(V)(0)$
is a subalgebra of $\mathcal{U}(V)$ and $L(0)\in\mathcal{U}(V)(0)$.

Claim. $W(m)$ is an irreducible $\mathcal{U}(V)(0)$-module, $\forall m\in\mathbb{Z}$.

For any nonzero elements $u^{1},\ u^{2}\in W(m),$ note that we have
$W=\langle v_{k}u^{1}\ |\ v\in V,\ k\in\mathbb{Z}\rangle.$ Recall
$X=\{v^{i}\ |\ i\in I\}$ is a basis of $V$ consisting of homogeneous
vectors. Let $o(v^{i})=v_{\mbox{wt}v^{i}-1}^{i}$ and extend notation
to all of $V$. Then we have \[
W(m)=\langle o(v)u^{1}\ |\ v\in V\rangle.\]
In particular, $u^{2}=o(v)u^{1}$ for some $v\in V$. Thus the claim
is proved.

Since $L(0)$ is in the center of $\mathcal{U}(V)(0)$ and $M(m)$
is an irreducible $\mathcal{U}(V)(0)$-module with countable dimension,
by Lemma \ref{-Schur's}, $L(0)$ acts on $W(m)$ as a constant, $\forall m\in\mathbb{Z}$.
Thus $L(0)$ acts semisimply on $W$.

\emph{Case 2}. If $\ell\left(w\right)>1$ for any $0\not=w\in W$.
Since $\ell(w)<\infty,\ \forall w\in W,$ we can define $L=\min\{\ell(w)\ |\ 0\not=w\in W\}$.

Assume $0\not=x=\sum_{j=1}^{L}x(i_{j})\in W$ where $x(i_{j})\in M(i_{j})$
and $i_{1}<\cdots<i_{L}$. Call $(i_{1},\ \cdots,\ i_{L})$ the pattern
of $x$. For any homogeneous $v\in V$ and $n\in\mathbb{Z},$ either
$v_{n}x=0$ or $0\not=v_{n}x$ has the pattern $(i_{1}+\mbox{wt}v_{n},\ \cdots,\ i_{L}+\mbox{wt}v_{n})$.
Let $K$ be a subset of $W$ consisting of vectors with pattern $(i_{1},\ \cdots,\ i_{L})$
together with 0. It is easy to see that $K$ is a subspace of $W$.

Claim. $K$ is an irreducible $\mathcal{U}(V)(0)$-module.

Let $u,\ w$ be any nonzero elements in $K.$ We can write $u=u(i_{1})+\cdots u(i_{L})$,
$w=w(i_{1})+\cdots w(i_{L})$, where $\ u(i_{j}),\ w(i_{j})\in M(i_{j})$,
$1\le j\le L$. Since $W=\mathcal{U}(V)u$, there exists $a^{n_{i}}\in\mathcal{U}(V)(n_{i}),$
$1\le i\le T$, $n_{1}<n_{2}<\cdots<n_{T}$ such that $w=a^{n_{1}}u+\cdots+a^{n_{T}}u$.
Note that $a^{n_{j}}u$ has pattern $(i_{1}+n_{j},\ \cdots,\ i_{L}+n_{j}).$
Clearly, the term $a^{n_{T}}u(i_{L})$ in $a^{n_{T}}u=a^{n_{T}}u(i_{1})+\cdots+a^{n_{T}}u(i_{L})$
is zero if $n_T>0.$ This forces $a^{n_{T}}u=0$ if $n_T>0.$ Otherwise, we have a nonzero
vector in $W$ whose length is less than $L$, a contradiction. Similarly,
if $n_{1}<0,$ then $a^{n_{1}}u=0.$ Continuing in this way shows
that there exists an $i$ such that some $n_{i}=0$ and $w=a^{0}u$
. This implies that $K$ is an irreducible $\mathcal{U}(V)(0)$-module.

Since $L(0)$ is in the center of $\mathcal{U}(V)(0)$ and $K$ is
an irreducible $\mathcal{U}(V)(0)$-module with countable dimension,
by Lemma \ref{-Schur's}, $L(0)$ acts on $K$ as a constant. Since
$W$ is generated by $u\in K,$ $L(0)$ acts semisimply on $W.$ \end{proof}

The following corollary is immediate.

\begin{corollary} The complete reducibility of any $\mathbb{Z}$-graded weak $V$-module is equivalent
to the semisimplicity of the $\mathbb{Z}$-graded weak $V$-module category. In particular, the complete reducibility of any $\mathbb{Z}_+$-graded weak $V$-module is equivalent to
the semisimplicity of the admissible $V$-module category.
\end{corollary}

We are now in the position to prove the main theorem of this paper.

\begin{theorem} \label{Thm-Z-graded S.S.->regular}For a vertex operator
algebra $V$, if any $\mathbb{Z}$-graded weak $V$-module 
is completely reducible, then $V$ is regular. \end{theorem}

\begin{proof}  As mentioned before, we only need to prove that $V$
is $C_{2}$-cofinite. Recall from {[}FHL{]} that the graded dual $V^{'}=\oplus_{n=0}^{\infty}V_{n}^{\ast}$
of $V$ is also an irreducible $V$-module. In particular, $V'$ is
a $\mathcal{L}(V)$-module. We extend the action of $\mathcal{L}(V)$
from $V'$ to $V^{\ast}=\prod_{n\in\mathbb{Z}}V_{n}^{\ast}$ in an
obvious way to make $V^{\ast}$ an $\mathcal{L}(V)$-module. Following
{[}L{]}, we set \[
D(V')=\{f\in V^{\ast}\ |\ u_{n}f=0,\ \forall u\in V,\ n\gg0\}.\]
Then $D(V')$ is the unique maximal weak module inside $V^{\ast}$
{[}L{]}. According to {[}L{]}, if $V'=D(V')$ then $V$ is $C_{2}$-cofinite.
We now prove that $V'=D(V')$.

If $V'\not=D(V')$, then there exists $f=(f_{n})_{n\in\mathbb{Z}}\in D(V')$
such that there are infinitely many nonzero components of $f$. Since
$L(0)f_{n}=nf_{n},$ $\forall n\in\mathbb{Z}.$ $f$ is not contained
in any finite dimensional $L(0)$-invariant subspace of $D(V')$.
Let $W$ be the weak submodule of $D(V')$ generated by $f$. It follows
from Corollary \ref{single v. weak module Z graded} that $W$ is
a quotient of a $\mathbb{Z}$-graded weak $V$-module $M$. From the
assumption that any $\mathbb{Z}$-graded module is completely reducible,
we see that $W$ is a submodule of a $\mathbb{Z}$-graded weak module $M.$
By Proposition \ref{Prop}, $L(0)$ is semisimple on $M$. In particular,
$L(0)$ is semisimple on $W$. This implies that $f\in W$ is contained
in a finite dimension $L(0)$-invariant subspace. This is a contradiction.
The proof is complete. \end{proof}

\begin{corollary} Let $V$ be a vertex operator algebra, then the
following statements are equivalent:

1) The weak $V$-module category is semisimple,

2) Any $\mathbb{Z}$-graded weak $V$-module is completely reducible,

3) $V$ is regular. \end{corollary}

\begin{proof} By Theorem \ref{Thm-Z-graded S.S.->regular}, 2) implies
3). The implication of 1) from 3) is obvious. We only need to show
that 1) implies 2). So it is enough to prove that any irreducible
weak $V$-module is $\mathbb{Z}$-graded. But this is clear from Proposition
\ref{Prop}. \end{proof}

\end{document}